\documentclass[11pt]{amsart}

\usepackage{amssymb,upref}
\usepackage[mathcal]{euscript}
\usepackage[cmtip,all]{xy}

\newcommand{\bydef}{:=}

\newcommand{\cA}{\mathcal{A}}

\newcommand{\cE}{{\mathcal E}}

\newcommand{\cI}{\mathcal{I}}

\newcommand{\cJ}{\mathcal{J}}

\newcommand{\cN}{{\mathcal N}}

\newcommand{\espan}[1]{\mathrm{span}\left\{#1\right\}}

\newcommand{\NN}{{\mathbb N}}

\newcommand{\FF}{\mathbb{F}}

\DeclareMathOperator{\Aut}{\mathrm{Aut}}

\DeclareMathOperator{\supp}{\mathrm{supp}\,}

\DeclareMathOperator{\ann}{ann}

\newcommand{\GL}{\mathrm{GL}}

\newtheorem{theorem}{Theorem}[section]
\newtheorem{proposition}[theorem]{Proposition}
\newtheorem{lemma}[theorem]{Lemma}
\newtheorem{corollary}[theorem]{Corollary}

\theoremstyle{definition}
\newtheorem{df}[theorem]{Definition}
\newtheorem{example}[theorem]{Example}

\theoremstyle{remark}
\newtheorem{remark}[theorem]{Remark}

\newenvironment{romanenumerate}
 {\begin{enumerate}
 
 }{\end{enumerate}}

\usepackage{tikz}
\usetikzlibrary{arrows}

\begin{document}

\title{Evolution  algebras and Graphs}

\author[Alberto Elduque]{Alberto Elduque$^{\star}$}
\thanks{$^{\star}$Supported by the Spanish Ministerio de Educaci\'on y Ciencia and FEDER (MTM 2010-18370-C04-02) and by the Diputaci\'on General de Arag\'on (Grupo de Investigaci\'on de \'Algebra).
Part of this research was done while this author was visiting the Departamento de Matem\'aticas, Facultad de Ciencias, Universidad de
Chile, supported by the FONDECYT grant 1120844.}
\address{Departamento de Matem\'aticas e Instituto Universitario de Matem\'aticas y Aplicaciones,
Universidad de Zaragoza, 50009 Zaragoza, Spain
}
\email{elduque@unizar.es}

\author[Alicia Labra]{Alicia Labra$^{\star\star}$}
\thanks{$^{\star\star}$Supported by FONDECYT 1120844.}
\address{Departamento de Matem\'aticas,
Facultad de Ciencias, Universidad de Chile.  Casilla 653, Santiago, Chile}
\email{alimat@uchile.cl}




\begin{abstract}
A digraph is attached to any evolution algebra. This graph leads to some new purely algebraic results on this class of algebras and allows for some new natural proofs of known results.
Nilpotency of an evolution algebra will be proved to be equivalent to the nonexistence of oriented cycles in the graph. Besides, the automorphism group of any evolution algebra $\cE$ with $\cE=\cE^2$ will be shown to be always finite.
\end{abstract}

\maketitle


\section{Introduction}

Evolution algebras were introduced in 2006 by Tian and Vojtechovsky in their paper ``Mathematical
concepts of evolution algebras in non-Mendelian genetics " (see \cite{TV}).
Later on, Tian laid the foundations of evolution algebras in his monograph
\cite{T}.
These algebras present many connections with other
mathematical fields including graph theory, group theory, Markov chains, dynamical systems, knot theory,
$3$-manifolds and the study of the Riemann-Zeta function (see \cite{T}).

Evolution algebras are not defined by identities, and hence they do not form a variety of non-associative algebras, like Lie, alternative or Jordan algebras. Therefore, the research on these algebras follows different paths (see \cite{CGOT}, \cite{CLOR1}, \cite{CLOR2}, \cite{RT}, \cite{TZ}).

In \cite[\S 6.1]{T}, Tian defined an evolution algebra associated to any digraph (directed graph). This lead him to study properties of what he called \emph{graphicable algebras}.

Here we will go in the reverse direction. Given any evolution algebra and a natural basis for it, a digraph with a weight assigned to its edges (weighted digraph) will be defined. If one forgets about the weights, this gives just a digraph. Properties of the evolution algebra will be related to corresponding properties in the digraph. The use of these properties simplifies and sheds some new light on existing results, and it allows for the discovery of new purely algebraic results for these algebras.

\smallskip

To simplify notations, all the graphs considered here will be directed graphs. These are pairs $(V,E)$ consisting on a set of vertices $V$, and a set of edges (or arrows) $E$, which is a subset of the cartesian product $V\times V$.

All our algebras will be defined over an arbitrary ground field $\FF$ and will have finite dimension. By algebra we mean just a vector space $\cA$ over our ground field $\FF$ endowed with a bilinear multiplication $\cA\times\cA\rightarrow \cA:\ (x,y)\mapsto xy$.

\smallskip

In Section \ref{se:evolution_and_graphs} the graph and weighted graph attached to any natural basis of an evolution algebra will be defined and some examples will be given. The geometrical condition of the graph being connected will be related to the condition on the algebra to split into a direct sum of simple ideals.

Then in Section \ref{se:nilpotency}, the result on nilpotency of evolution algebras in \cite[Theorem 2.7]{CLOR1} will be revised and expanded. The main result is that nilpotency can be read from the graph (see Theorem \ref{th:nilpotency} and Remark \ref{re:nilpotency_graph}): an evolution algebra is nil, or nilpotent, if and only if the graph contains no oriented cycle. This point of view simplifies too some of the arguments in \cite{CLOR1}.

Finally, in Section \ref{se:automorphisms} we will deal with automorphisms of evolution algebras. These have been considered in \cite{CGOT}. The main result asserts that the automorphism group of any evolution algebra $\cE$ such that $\cE=\cE^2$ is finite. An example will be given to show that the condition $\cE=\cE^2$ is necessary.

\bigskip

\section{Evolution algebras and Graphs}\label{se:evolution_and_graphs}

\begin{df}\cite{T}\label{df:evolution}
An  \emph{evolution algebra} is an algebra $\cE$ containing a basis (as a vector space) $B=\{e_1, \cdots, e_n\}$ such that $e_ie_j=0$ for any $1\leq i<j\leq n$. A basis with this property is called a \emph{natural basis}.
\end{df}

By its own definition, any evolution algebra is commutative.

Given a natural basis $B=\{e_1,\ldots,e_n\}$ of an evolution algebra $\cE$,
\[
e_i^2=\sum_{j=1}^n\alpha_{ij}e_j
\]
for some scalars $\alpha_{ij}\in \FF$, $1\leq i,j\leq n$. The matrix $A=\bigl(\alpha_{ij}\bigr)$ is the \emph{matrix of structural constants} of the evolution algebra $\cE$, relative to the natural basis $B$.

\smallskip

We define next the graph and weighted graph attached to an evolution algebra. Recall that our graphs are always directed graphs.

\begin{df}
Let $\cE$ be an evolution algebra with a natural basis $B=\{e_1,\ldots,e_n\}$ and matrix of structural constants $A=\bigl(\alpha_{ij}\bigr)$.
\begin{itemize}
\item The graph $\Gamma(\cE,B)=(V,E)$, with $V=\{1,\ldots,n\}$ and $E=\{(i,j)\in V\times V: \alpha_{ij}\ne 0\}$, is called the \emph{graph attached to the evolution algebra $\cE$ relative to the natural basis $B$}.
\item The triple $\Gamma^w(\cE,B)=(V,E,\omega)$, with $\Gamma(\cE,B)=(V,E)$ and where $\omega$ is the map $E\rightarrow \FF$ given by $\omega\bigl((i,j)\bigr)=\alpha_{ij}$, is called the \emph{weighted graph attached to the evolution algebra $\cE$ relative to the natural basis $B$}.
\end{itemize}
\end{df}

Let us see a few examples.

\begin{example}
Let $\cE$ be an evolution algebra with natural basis $B = \{e_1,e_2\}$ and multiplication given by $e_1^2 = \alpha e_1+ \beta e_2$, $e_2^2 = \gamma e_1 + \delta e_2$, with $\alpha,\beta,\gamma,\delta \in \FF^\times$. Then the graph $\Gamma^w(\cE, B)$ is
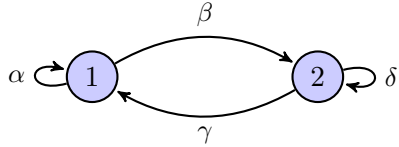
\[
\begin{tikzpicture}%
[->,>=stealth',shorten >=1pt,auto,thick,every node/.style={circle,fill=blue!20,draw}]
  \node (n1) at (4,8)  {1};
  \node (n2) at (7,8) {2};

   \path[every node/.style={font=\sffamily\small}]
           (n2)  edge [bend left] node[below] {$\gamma$} (n1)
       (n2) edge [loop right] node {$\delta$} (n2)
        (n1)  edge  [bend left] node[above] {$\beta$} (n2)
       (n1) edge [loop left] node {$\alpha$} (n1);

\end{tikzpicture}
\]
\end{example}

\begin{example}\label{ex:2bases_graph}
Let $\cE$ be an evolution algebra with natural basis $B = \{e_1,e_2\}$ and multiplication given by
$e_1^2 = e_2$ and $e_2^2 = e_2$. Then the graph $\Gamma^w(\cE, B) $ is
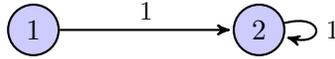
\[
\begin{tikzpicture}%
[->,>=stealth',shorten >=1pt,auto,thick,every node/.style={circle,fill=blue!20,draw}]
  \node (n1) at (4,8)  {1};
  \node (n2) at (7,8) {2};

   \path[every node/.style={font=\sffamily\small}]

       (n2) edge [loop right] node {$1$} (n2)
       (n1) edge node [above] {$1$} (n2);

\end{tikzpicture}
\]
Note that $(e_1 +e_2)(e_1 - e_2) = 0$ and hence, if the characteristic of $\FF$ is $\ne 2$, $B' = \{f_1 = e_1 +e_2, f_2 = e_1 - e_2\}$ is another natural basis of $\cE$, with $f_1^2 = f_1 -f_2$ and $f_2^2 = f_1 -f_2$. The weighted graph $\Gamma^w(\cE, B') $ relative to $B'$ is
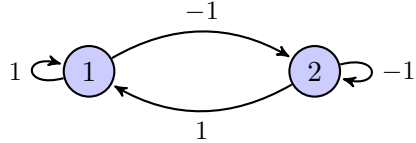
\[
\begin{tikzpicture}%
[->,>=stealth',shorten >=1pt,auto,thick,every node/.style={circle,fill=blue!20,draw}]
  \node (n1) at (4,8)  {1};
  \node (n2) at (7,8) {2};

   \path[every node/.style={font=\sffamily\small}]

           (n2)  edge [bend left] node[below] {$1$} (n1)
        edge [loop right] node {$-1$} (n2)
        (n1)  edge  [bend left] node[above] {$-1$} (n2)
       (n1) edge [loop left] node {$1$} (n1) ;

\end{tikzpicture}
\]
\end{example}

\begin{example}\label{ex:connected_notconnected}
Let $\cE$ be an evolution algebra and $B = \{e_1,e_2, e_3\}$ a natural basis with
$e_1^2 = e_2 + e_3$, $e_2^2 = 0$  and $e_3^2 = 0.$ Then the graph $\Gamma^w(\cE, B) $ is
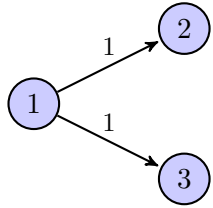
\[
\begin{tikzpicture}%
[->,>=stealth',shorten >=1pt,auto,thick,every node/.style={circle,fill=blue!20,draw}]

  \node (n1) at (4,4)  {1};
  \node (n2) at (6,5)  {2};
  \node (n3) at (6,3)  {3};

   \path[every node/.style={font=\sffamily\small}]

           (n1) edge node [above] {$1$} (n2)
       (n1) edge node [above] {$1$} (n3) ;

\end{tikzpicture}
\]
so $\Gamma(\cE,B)$ is a connected graph. However,
$B' = \{f_1 = e_1, f_2 =e_2 +e_3, f_3 = e_3\}$ is another natural basis of $\cE$ with $f_1^2 = f_2$ and  $f_2^2 = 0 = f_3^2$.
The new graph $\Gamma^w(\cE, B') $ is
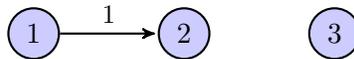
\[
\begin{tikzpicture}%
[->,>=stealth',shorten >=1pt,auto,thick,every node/.style={circle,fill=blue!20,draw}]

  \node (n1) at (4,4)  {1};
  \node (n2) at (6,4)  {2};
  \node (n3) at (8,4)  {3};

   \path[every node/.style={font=\sffamily\small}]

           (n1) edge node [above] {$1$} (n2) ;

\end{tikzpicture}
\]
so $\Gamma(\cE,B')$ is not connected!
\end{example}

\medskip

Contrary to the definition of \emph{evolution ideal} in \cite[\S 3.1.4]{T}, the word \emph{ideal} will carry here the usual meaning. That is,  an ideal $\cI$ of an evolution algebra $\cE$ is a subspace satisfying $\cE\cI\subseteq \cI$. No other condition will be imposed.

\begin{df}
An algebra $\cE$ is said to be \emph{decomposable} if there are nonzero ideals $\cI$ and  $\cJ$ such that $\cE = \cI \oplus \cJ$. Otherwise, $\cE$ is called \emph{indecomposable}.
\end{df}

Connectedness is related to indecomposability.

Given an evolution algebra $\cE$, consider its \emph{annihilator ideal} $\ann(\cE)\bydef\{x\in\cE: x\cE=0\}$.

\begin{lemma}\label{le:annihilator}
Let $B=\{e_1, \cdots, e_n\}$ be a natural basis of an evolution algebra $\cE$. Then
\[
\ann(\cE)  = \espan{e_i : e_i^2 = 0}.
\]
\end{lemma}
\begin{proof}
If $ 0 \neq x = \alpha_1 e_1 + \cdots  + \alpha_n e_n $ is an element such that $x \cE = 0$, for any $i$ with  $ \alpha_i \neq 0 $,  $ 0 = xe_i = \alpha_i e_i^2$, so $e_i^2=0$.

Conversely, if  $e_i^2 = 0$ for some $i$, then $e_i \cE = 0$, because  $e_i e_j   =0$ for any $j\ne i$.
\end{proof}

\begin{proposition}\label{pr:indecomposable_connected}
Let $\cE$ be an evolution algebra with natural basis $B=\{e_1,\ldots,e_n\}$ and such that $\ann(\cE) = 0$.
Then $\cE$ is indecomposable if and only if $\Gamma(\cE, B)$ is connected.
\end{proposition}
\begin{proof}
Let us prove that $\cE$ is decomposable if and only if  $\Gamma(\cE,B)$ is not connected.

If $\cE$ is decomposable: $\cE=\cI\oplus\cJ$ for nonzero ideals $\cI$ and $\cJ$, then for each $i=1,\ldots,n$, $e_i = e_i' + e_i''$, with  $e_i'\in\cI$ and $e_i''\in \cJ$.  Thus $e_i^2 = e_i'^2  + e_i ''^2$ and $ \sum_{j=1}^n \alpha_{ij}e_j = \sum_{j=1}^n \alpha_{ij}(e_j' + e_j'') = \sum_{j=1}^n \alpha_{ij}e_j' + \sum_{j=1}^n \alpha_{ij}e_j''$.
Therefore, for  every $i=1,\ldots,n$, $e_i'^2 = \sum_{j=1}^n \alpha_{ij}e_j'$ and $e_i''^2 = \sum_{j=1}^n \alpha_{ij}e_j''$.

Moreover, for every $i \neq j$, $e_i  e_j = 0$ and hence $e_i'  e_j' = 0 $ and $e_i''  e_j'' = 0$ too.

The set $\{ e_i' :  1 \leq i \leq n \}$ spans $\cI$ and  the set $\{ e_i'' : 1 \leq i \leq n \}$ spans $\cJ$.
If for some $ 1 \leq j \leq n$,  $e_j' \in \FF e_1' + \cdots + \FF e_{j-1}' $ (for $j=1$ this means $e_1'=0$) then $e_j'^2 \in e_j'(\FF e_i' + \cdots + \FF e_{j-1}') = 0$. Moreover,
$e_j'  e_h' = 0$ for all $j \neq h$, so $e_j' \cI = 0$. But $e_j'\cJ\subseteq\cI\cJ=0$, so $e_j' \cE = 0$ and $e_j'\in\ann(\cE)=0$.

Define $ I \bydef \{i : 1\leq i \leq n, \; e_i' \neq 0  \} $ and $ J \bydef \{ j : 1\leq j \leq n, \; e_j'' \neq 0 \}$. Clearly $\{1,\ldots,n\}=I\cup J$. The argument above shows that $\{ e_i' : i \in I\}$ and $\{ e_j'' : j \in J\}$ are $\FF$-bases of $\cI$  and
$\cJ$ respectively. Moreover, $|I| + |J| = \dim(\cI) + \dim(\cJ) = n$, and this forces $I$ and $J$ to be disjoint. Hence $B=B_\cI\cup B_\cJ$ (disjoint union), where $B_\cI=\{e_i: i\in I\}$ is a (natural) basis of $\cI$ and $B_\cJ=\{e_i:i\in J\}$ is a (natural) basis of $\cJ$. The vertices of $\Gamma(\cE,B)$ in $I$ are not connected with the vertices in $J$, because $\cI^2\subseteq \cI$ and $\cJ^2\subseteq \cJ$, and hence $\Gamma(\cE,B)$ is not connected.

Conversely, suppose that  $\Gamma(\cE, B)$ is not connected. Then there exists a partition $\{1, \cdots,n\} = I \cup J $ (disjoint union) such that there is no edge connecting an element in $I$ with an element in $J$.  Then $\cI\bydef \espan{e_i :  i  \in I}$ and $ \cJ = \espan{e_j :  j  \in J}$ are ideals of
$\cE$ and $\cE = \cI \oplus \cJ$. Therefore, $\cE$ is decomposable.
\end{proof}

\smallskip

In general, an ideal of an evolution algebra is not an evolution algebra itself. However, the arguments in the proof above show that quotients behave nicely.

\begin{lemma}\label{le:quotient}
If $\cE$ is an evolution algebra and $\cI$ is a proper ideal of $\cE$, then $\cE / \cI$ is an evolution algebra.
\end{lemma}
\begin{proof}
Take a natural basis $B= \{e_1, \cdots, e_n\}$ of $\cE$. Then the set of coclasses module $\cI$ of elements in $B$: $B' = \{e_i' = e_i + \cI : i, \; 1 \leq  i \leq n \}$,
spans $\cE / \cI$, and $e_i' e_j' = 0$ for $ i \neq j.$ Any basis of $\cE / \cI$ contained in $B'$
is then a natural basis of $\cE / \cI$.
\end{proof}

(Note that in the proof above, the elements of $B'$ not in the chosen basis, are necessarily in $\ann(\cE/\cI)$, because the product of any of these elements by any element of the basis is trivial.)

\smallskip

In case the annihilator of an evolution algebra is not trivial, Proposition \ref{pr:indecomposable_connected} is no longer valid. Actually, as shown by Example \ref{ex:connected_notconnected}, the property of the graph being connected depends on the chosen natural basis. The right result in this case is the following proposition.

\begin{proposition}
Let $\cE$ be an evolution algebra.
Then $\cE$ is indecomposable if and only if the graph $\Gamma(\cE,B)$ is connected for any natural basis $B$.
\end{proposition}
\begin{proof}
Let us prove that $\cE$ is decomposable if and only if  there is a natural basis of $\cE$ such that $\Gamma(\cE, B)$ is not connected.

If $\cE$ is decomposable then there are non zero ideals $\cI, \cJ$ such that $\cE = \cI \oplus \cJ.$ Then $\cI \cong
\cE / \cJ$ and $\cJ \cong \cE / \cI$ are evolution algebras by Lemma \ref{le:quotient}. Take $B_{\cI}$ a natural basis of $\cI$ and
$B_{\cJ}$ a natural basis of $\cJ$. Then $B = B_{\cI} \cup B_{\cJ}$ is a natural basis of $\cE$ and  $\Gamma(\cE, B)$ is not connected, because there is no edge connecting nodes corresponding to $B_{\cI}$ with nodes corresponding to  $ B_{\cJ}$. The converse is proven as in Proposition \ref{pr:indecomposable_connected}.
\end{proof}

\bigskip

\section{Nilpotency}\label{se:nilpotency}

The goal of this section is to reprove and extend \cite[Theorem 2.7]{CLOR1} on the nilpotency of evolution algebras. Let us recall first the definitions.

\begin{df}
An element $x$ of an evolution algebra $\cE$  is called \emph{nil} if
there is a natural number $n$ such that  $(\cdot
\underbrace{((x\cdot x)\cdot x)\cdots x}_{n})=0$.
The algebra $\cE$ is said to be \emph{nil} if every  element of the algebra
is nil.
\end{df}

Given an evolution algebra, we introduce the following sequences of subspaces:
\begin{align*}
\cE^{<1>} & = \cE,&  \cE^{<k+1>} & = \cE^{<k>}\cE;\\
\cE^1 & = \cE, & \cE^{k+1} &=\sum_{i=1}^{k}\cE^i\cE^{k+1-i}.
\end{align*}

\begin{df} An algebra $\cE$ is called
\begin{romanenumerate}
\item \emph{right nilpotent} if there exists $n\in \NN$ such that $\cE^{<n>} = 0$, and the minimal such number is called the \emph{index of right nilpotency};
\item \emph{nilpotent} if there exists $n\in \NN$ such that $\cE^n = 0$, and the minimal such number is called the \emph{index of nilpotency}.
\end{romanenumerate}
\end{df}

\begin{remark}
A commutative algebra is right nilpotent if and only if it is nilpotent (see \cite[Chapter 4, Proposition 1]{ZSSS}). This applies, in particular, to evolution algebras.
\end{remark}

\medskip

\begin{theorem}\label{th:nilpotency}
Let $\cE$ be an evolution algebra and let  $ B = \{e_1, \cdots, e_n\}$ be a natural basis of $\cE$. Then the following condition are equivalent:
\begin{enumerate}
\item $\cE$ is nil.
\item There are no oriented cycles in $\Gamma(\cE, B)$.
\item The basis $B$ can be reordered so that the matrix of structural constants $A$ is strictly upper triangular.
\item $\cE$ is nilpotent.
\end{enumerate}
\end{theorem}
\begin{proof}

$(4)\Longrightarrow  (1)$ is trivial.

$(1)\Longrightarrow  (2)$: Suppose, on the contrary, that $\Gamma(\cE, B)$ contains oriented cycles and choose one of minimal
length (this length may be $1$). Let $r$ be this minimal length. Reorder the basic elements so that this cycle is
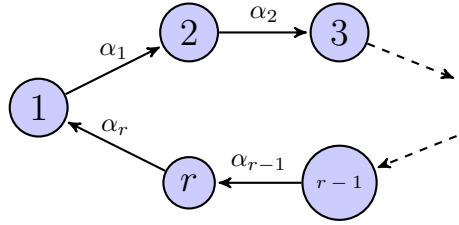
\[
\begin{tikzpicture}%
[->,>=stealth',shorten >=1pt,auto,thick,main node/.style={circle,fill=blue!20,draw}]

  \node[main node] (n1) at (4,4)  {\Large{$1$}};
  \node [main node](n2) at (6,5)  {\Large{$2$}};
    \node [main node](n3) at (8,5)  {\Large{$3$}};
    \node (n4) at (10.5,4)  {\Huge{\phantom{$4444$}}};
      \node [main node](nr) at (8,3)  {\tiny{$r-1$}};
  \node [main node] (nrr) at (6,3)  {\Large{$r$}};

   \path[every node/.style={font=\sffamily\small}]

           (n1) edge node [above] {$\alpha_1$} (n2)
     (n2) edge node [above] {$\alpha_2$} (n3)
         (n3) edge [dashed] node [above] {} (n4)
              (nrr) edge node [above] {$\alpha_r$}  (n1)
                     (nr) edge node [above] {$\alpha_{r-1}$}  (nrr)
	(n4) edge [dashed] node {} (nr);

\end{tikzpicture}
\]

\medskip
By minimality there is no other arrow connecting the nodes $1, \cdots, r$. Therefore, if $\cN = \FF e_{r+1} \oplus \cdots \oplus \FF e_n$, then
\[
\begin{split}
e_1^2&= \alpha_1 e_2 + u_1, \quad u_1 \in \cN,\\
e_2^2&= \alpha_2 e_3 + u_2, \quad u_2 \in \cN,\\
&\vdots\\
e_{r}^2 &= \alpha_r e_1 + u_r, \quad u_{r} \in \cN.
\end{split}
\]
Take the element $x = e_1 + \cdots + e_{r}$. Write $u \equiv v$ if and only if $ u-v \in \cN$.  Then $x^{r+1} \equiv (\alpha_1\alpha_2 \cdots \alpha_r)x$,
so $x$ is not nilpotent and $\cE$ is not nil.

$(2)\Longrightarrow  (3)$: Since there are no oriented cycles, there is a sink (i.e., a node with no arrow leaving from it). If this
node is $i$,  this means $e_i^2 = 0$. Reordering $B$ we may assume that $ i = n$, so $ e_n^2 = 0$.  Now the subgraph consisting of
the nodes $ 1, \ldots, n-1$ and the arrows connecting them contains no oriented cycle, so we may assume now that the node $n-1$ is a sink.
Therefore, $e_{n-1}^2 \in \FF e_n$.  Continuing in this way,  we may reorder our natural basis $B$ so that  $e_{n-2}^2 \in \FF e_{n-1} + \FF e_n$, ...
Eventually we reorder the basis so that $e_i^2 \in \sum_{j>i}\FF e_j$ for any $i$, and hence $A$ is strictly upper triangular.

$(3)\Longrightarrow  (4)$: If $A$ is strictly upper triangular, then $\cE^{<k>} \subseteq \sum_{j=k}^n\FF e_j$. In particular, $\cE^{<n + 1>} = 0$.
\end{proof}

\begin{remark}\label{re:nilpotency_graph}
Nilpotency can be read from the graph!!
\end{remark}

The equivalence of (1), (3) and (4) appears in \cite[Theorem 2.7]{CLOR1} (see also \cite{TZ}). Note that the use of  $\Gamma(\cE, B)$ simplifies
and sheds new light on the proof.


\bigskip

\section{Automorphisms}\label{se:automorphisms}

This last section is devoted to prove that the group of automorphisms of any evolution algebra with $\cE=\cE^2$ is finite. This happens, in particular, if $\cE$ is simple or a direct sum of simple ones, and it shows how rigid these algebras are.

Automorphisms of evolution algebras have been considered too in \cite{CGOT}.

First, note that if we drop the condition $\cE=\cE^2$, then $\Aut(\cE)$ may be infinite, even if $\ann(\cE)=0$, as shown by the next example.

\begin{example}
Let $\cE$ be the evolution algebra over the field of real numbers with natural basis $B=\{e_1,e_2,e_3\}$ and attached graph
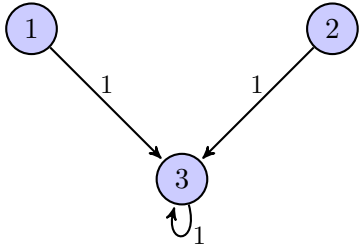
\[
\begin{tikzpicture}%
[->,>=stealth',shorten >=1pt,auto,thick,every node/.style={circle,fill=blue!20,draw}]

  \node (n1) at (2,4)  {1};
  \node (n2) at (6,4)  {2};
  \node (n3) at (4,2)  {3};

   \path[every node/.style={font=\sffamily\small}]

           (n1) edge node [above] {$1$} (n3)
       (n2) edge node [above] {$1$} (n3)
       (n3) edge [loop below] node[right] {$1$} (n3);

\end{tikzpicture}
\]
A straightforward computation shows that a linear automorphism $\varphi\in \GL(\cE)$ is an automorphism if and only if it fixes $e_3$ and satisfies $e_1\mapsto \alpha e_1+\beta e_2$, $e_2\mapsto\gamma e_1+\delta e_2$, with $\left(\begin{smallmatrix} \alpha&\beta\\ \gamma &\delta\end{smallmatrix}\right)$ an orthogonal matrix.  Hence the group of automorphisms $\Aut(\cE)$ is infinite. Lemma \ref{le:annihilator} shows  $\ann(\cE)=0$.
\end{example}

\begin{proposition}
Let $\cE$ be an evolution algebra and let $A$ be its matrix of structural constants relative to a natural basis $B=\{e_1,\ldots,e_n\}$. Then $\cE=\cE^2$ if and only if $A$ is regular ($\det(A) \neq 0$).
\end{proposition}
\begin{proof}
For any $i=1,\ldots,n$,
$e_i^2 =  \sum_{j=1}^n \alpha_{ij} e_j$, so $A$ is regular if and only if $e_1^2, \ldots, e_n^2$ are linearly independent, and this happens if and only if
$\cE = \cE^2$. (Note that $e_1^2,\ldots,e_n^2$ span $\cE^2$.)
\end{proof}

Let $B=\{e_1, \cdots, e_n\}$ be a natural basis of an evolution algebra $\cE$.  For any element $ x = \alpha_1 e_1 + \cdots  + \alpha_n e_n $ define its \emph{support} (relative to $B$) by
$\supp(x) \bydef \{i : \alpha_i \neq 0\}$.

\begin{lemma}
Let $B=\{e_1, \cdots, e_n\}$ be a natural basis of an evolution algebra $\cE$ satisfying $\cE=\cE^2$, and let $ 0 \neq x, y \in \cE $ be two nonzero elements such that $xy = 0$. Then $\supp(x) \cap \supp(y) = \emptyset$.
\end{lemma}

\begin{proof}
Let $0 \neq x,y $ be elements in $\cE$. Then $x = \alpha_1 e_1 + \cdots  + \alpha_n e_n$, $y = \beta_1 e_1 + \cdots  + \beta_n e_n$ and $ 0 = xy =\alpha_1 \beta_1 e_1^2 + \cdots  + \alpha_n \beta_n e_n^2$. But $\cE=\cE^2$, so the elements $e_1^2,\ldots,e_n^2$ are linearly independent, and hence for any $i=1,\ldots,n$ we get $\alpha_i \beta_i = 0$. Hence for any $i$ either  $\alpha_i  = 0$ or $\beta_i = 0$ (or both) and $\supp(x) \cap \supp(y) = \emptyset$.
\end{proof}

Before proving our result on the finiteness of the automorphism group, we need the next easy result, which has its own independent interest.

\begin{theorem}
Let $\cE$ be an evolution algebra such that $\cE=\cE^2$, and let $B=\{e_1,\ldots,e_n\}$ and $B'=\{f_1,\ldots,f_n\}$ be two natural bases. Then
there exists a permutation $\sigma \in S_n$ such that for any $i=1,\ldots,n$, $f_i \in \FF^{\times} e_{\sigma(i)}$.
\end{theorem}

\begin{proof}
For any $i\ne j$, $f_i f_j = 0$, so that $\supp(f_i) \cap \supp(f_j) =  \emptyset$, where $\supp$ indicates the support relative to the natural basis $B$. Then, necessarily, $\supp(f_i)$ consists of a single element for any $i$, and the result follows.
\end{proof}

\begin{corollary}
Let $\cE$ be an evolution algebra such that $\cE=\cE^2$, then the isomorphism class of the graph $\Gamma(\cE,B)$ does not depend on the natural basis $B$.
\end{corollary}

\begin{remark}
The condition $\cE=\cE^2$ is indeed necessary, as shown by Examples \ref{ex:2bases_graph} and \ref{ex:connected_notconnected}.
\end{remark}

\begin{corollary}\label{co:aut}
Let $\cE$ be an evolution algebra such that $\cE=\cE^2$ and let $B=\{e_1,\ldots,e_n\}$ be a natural basis. Then for any automorphism $\varphi\in\Aut(\cE)$ there is a permutation $\sigma\in S_n$ such that $\varphi(e_i)\in\FF^\times e_{\sigma(i)}$.
\end{corollary}

\begin{theorem}
Let $\cE$ be an evolution algebra such that $\cE=\cE^2$. Then its group of automorphisms $\Aut(\cE)$ is finite.
\end{theorem}
\begin{proof}
Fix a natural basis $B=\{e_1,\ldots,e_n\}$ of $\cE$, with matrix of structural constants $A=\bigl(\alpha_{ij}\bigr)$. By Corollary \ref{co:aut}, any automorphism of $\cE$ induces an automorphism of $\Gamma(\cE, B)$.
Let $\Phi: \Aut(\cE) \longrightarrow  \Aut(\Gamma(\cE, B))$ be the corresponding group automorphism. Moreover, $\Aut(\Gamma(\cE, B))$ is a finite group (up to isomorphism it is a subgroup of the symmetric group), so it is enough to prove that $\ker(\Phi)$ is finite. Notice that this kernel consists of the \emph{diagonal automorphisms}:
for any $\varphi \in \ker(\Phi)$ there are scalars $\mu_i \in \FF^\times$ such that   $\varphi(e_i) = \mu_i e_i$ for all $i$.

For such a diagonal linear map, the conditions on $\varphi$ to be an automorphism are given by $\varphi(e_i^2)= \varphi(e_i)^2$ for any $i$. But $\varphi(e_i)^2 = \mu_i^2  e_i^2 = \sum_{j} \mu_i^2 \alpha_{ij} e_j$. On the other hand, $\varphi(e_i^2) = \varphi( \sum_{j} \alpha_{ij} e_j) = \sum_{j} \alpha_{ij}\mu_j e_j$. Then $\varphi$ is an automorphism if and only if
$\mu_i^2 \alpha_{ij} = \mu_j \alpha_{ij}$ for all $i,j$ such that $\alpha_{ij}\ne 0$, or $\mu_j=\mu_i^2$ for any $i,j$ such that $\alpha_{ij}\ne 0$. This is equivalent to $\mu_j=\mu_i^2$ for any $(i,j)\in E$, so this is really a condition on the graph $\Gamma(\cE,B)=(V,E)$. We summarize this argument in the following group isomorphism:
\[
\ker(\Phi)\cong \{(\mu_1,\ldots,\mu_n)\in(\FF^\times)^n:   \mu_j = \mu_i^2\  \forall \; (i,j) \in E  \}.
\]
(The group on the right hand side is a subgroup of the $n$-dimensional torus $(\FF^\times)^n$.)

But $\cE = \cE^2$, so for any $i$, $e_i$ lies in $\cE^2$ and hence there exists an index $j$ such that $(j,i)\in E$. Therefore, for any  $i$, write $i_0=i$, and there is an index $i_1$ such that $(i_1,i_0)\in E$. Then there is an $i_2$ with $(i_2,i_1)\in E$, ...

Take the lowest integers $0\leq r<s<n$ such that $i_r=i_{s+1}$.
{\small
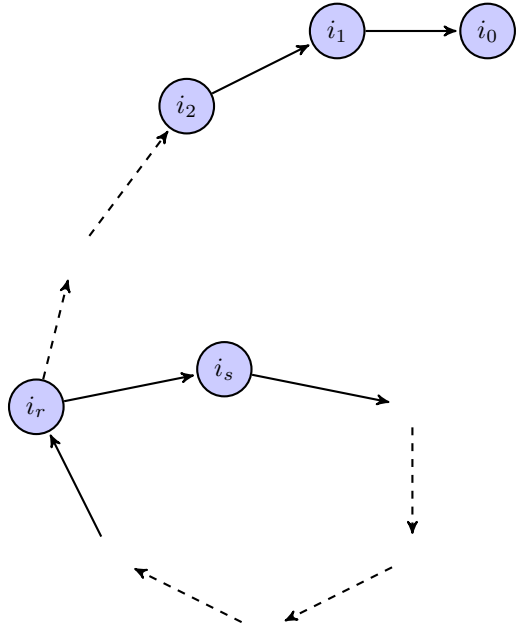
\[
\begin{tikzpicture}%
[->,>=stealth',shorten >=1pt,auto,thick,main node/.style={circle,fill=blue!20,draw}]
  \node [main node] (n0) at (7,9)  {$i_0$};
	\node [main node] (n1) at (5,9)  {$i_1$};
	\node [main node] (n2) at (3,8)  {$i_2$};
	\node [main node] (nr) at (1,4)  {$i_r$};
	\node [main node] (ns) at (3.5,4.5)  {$i_s$};

	 \node (n3) at (1.5,6)  {\phantom{$i_1$}};
	 \node (n4) at (6,4)  {\phantom{$i_1$}};
	 \node (n5) at (6,2)  {\phantom{$i_1$}};
	 \node (n6) at (4,1)  {\phantom{$i_1$}};
	 \node (n7) at (2,2)  {\phantom{$i_1$}};

   \path[every node/.style={font=\sffamily\small}]

       (n1) edge  node {} (n0)
	 (n2) edge  node {} (n1)
	 (n7) edge  node {} (nr)
	 (nr) edge  node {} (ns)
	 (ns) edge  node {} (n4)

       (n3) edge [dashed] node  {} (n2)
	 (nr) edge [dashed] node  {} (n3)
	 (n4) edge [dashed] node  {} (n5)
	 (n5) edge [dashed] node  {} (n6)
	 (n6) edge [dashed] node  {} (n7);

\end{tikzpicture}
\]}

\medskip

In this situation, we get
\[
\mu_{i_r} = \mu_{i_{r+1}}^2 = \mu_{i_{r+2}}^{2^2} = \cdots = \mu_{i_{s+1}}^{2^{s+1-r}} = \mu_{i_{r}}^{2^{s+1-r}},
\]
so we have $\mu_{i_{r}}^{2^{s+1-r}-1} = 1$. But also
\[
\mu_i = \mu_{i_0} = \mu_{i_{1}}^2 = \cdots = \mu_{i_{r}}^{2^{r}}.
\]
Replacing this value in the above
expression, we get
\[
\mu_i^{2^{s+1-r}-1} =  (\mu_{i_{r}}^{2^r})^{2^{s+1-r}-1}  = (\mu_{i_{r}}^{2^{s+1-r}-1})^{2^r} = 1.
\]
Therefore, $\mu_i$ is a root of unity of order a divisor of $2^{s+1-r}-1.$
This implies that there is only a finite number of possibilities for each $\mu_i$. Hence $\ker(\Phi) $ is  finite and the Theorem follows.
\end{proof}

\begin{remark}
The proof above leads easily to an algorithmic procedure to determine the automorphism group of any evolution algebra with $\cE=\cE^2$.
\end{remark}


\end{document}